\documentclass[12pt]{amsart}
\usepackage{amsmath,amsfonts,amsthm,amscd, amssymb,mathrsfs}
\usepackage{enumitem}
\usepackage{xcolor}
\usepackage[sort,nocompress]{cite}
\usepackage{hyperref}
\usepackage{tikz, tikz-cd}
\usetikzlibrary{calc}
\usepackage{bm}
\newtheorem{theorem}{Theorem}[section]

\newtheorem{proposition}[theorem]{Proposition}
\newtheorem{lemma}[theorem]{Lemma}
\newtheorem{corollary}[theorem]{Corollary}
\theoremstyle{definition}

\newtheorem{remark}[theorem]{Remark}

\newcommand{\PP}{\mathbb{P}}
\newcommand{\CC}{\mathbb{C}}
\newcommand{\RR}{\mathbb{R}}
\newcommand{\ZZ}{\mathbb{Z}}
\newcommand{\QQ}{\mathbb{Q}}

\newcommand{\Coker}{{\rm{Coker}}}

\newcommand{\red}{\mathrm{red}}

\DeclareMathOperator{\Ker}{Ker}
\DeclareMathOperator{\Hom}{Hom}
\DeclareMathOperator{\Image}{Image}

\usepackage[letterpaper]{geometry}
\numberwithin{equation}{section}
\begin{document}
\title[Logarithmic transformations on the six-sphere]{Logarithmic transformations on the six-sphere}
\author{Jeff Viaclovsky}
\address{Department of Mathematics, University of California, Irvine, USA}
\email{jviaclov@uci.edu}
\date{October 28, 2026.}
\begin{abstract}
We show that there are infinitely many choices of the multiplicities in the logarithmic transformation step of the Alp\"oge--Claude construction which produce a complex threefold diffeomorphic to the six-sphere. Moreover, distinct choices of the multiplicities yield biholomorphically inequivalent complex structures.
\end{abstract}
\maketitle
\tableofcontents

\section{Introduction}
In \cite{EngelArxiv}, Engel explained the construction of Alp\"oge--Claude \cite{AC}, 
starting from a rational elliptic surface $\pi: S \rightarrow \PP^1$ with 
singular fibers of type $\mathrm{IV}^*, \mathrm{III}, \mathrm{I}_1$
over the points $0, 1, \infty$, respectively. 
For a complex number $u$ in a sufficiently small punctured disc, Engel constructs a fibration $X^{\circ}(u)\to\PP^1\setminus\{\infty\}$ with smooth complex $2$-torus fibers away from $0$ and $1$. This can be filled in at infinity by a rank $2$ Mumford hexagonal degeneration to get a fibration $X(u) \rightarrow \PP^1$. 
The added singular fiber over infinity has $3$ double curves, $2$ triple points, and normalizes to a degree $6$ del Pezzo surface.  
The space $X(u)$ does not yet have the topology of $S^6$. The next step is to perform twisted logarithmic transformations to the fibration at $0$ and $1$, to replace  $X_0$  and $X_1$ with smooth bielliptic surfaces having multiplicities $3$ and $4$, respectively. This produces a fibration $X_{3,4}(u) \rightarrow \PP^1$.  For suitable choice of the logarithmic transformation data, the total space $X_{3,4}(u)$ is diffeomorphic to $S^6$. The main result in this paper is the following generalization of this construction.  
\begin{theorem} 
\label{t:main} 
Let $u \in \CC$ with $|u|$ sufficiently small and non-zero. Given positive integers $k$ and $l$ with $\mathrm{gcd}(3k, 4l) =1$, one can perform 
twisted logarithmic transformations of order $3k$ and $4l$ over the points $0$ and $1$, respectively, so that $X_{3k,4l}(u)$ is diffeomorphic to $S^6$. 
If $X_{3k,4l}(u)$ is biholomorphic to $X_{3k',4l'}(u')$, then $k = k'$ and $l =l'$.
Moreover, the local Kuranishi dimension for every such $X_{3k,4l}(u)$ is at least $1$. 
\end{theorem}
We note that $X_{3k,4l}(u)$ also depends upon choices of admissible translation vectors $a_0$ and~$a_1$.  We will refer to the $4$-tuple $(m_0,a_0, m_1, a_1)$ of choices of multiplicities and translation vectors as \textit{logarithmic transformation data}.

\begin{remark} One can use $X_{3k,4l}(u)$ as the starting point for the combs construction of the author in ~\cite{Combs} to produce even more new examples of complex structures on the smooth manifold $S^6$ with arbitrarily large local Kuranishi dimension.  Blowing up at a point yields a similar phenomenon on the smooth manifold underlying $\PP^3$.  This is a relatively straightforward extension, but to keep this paper focused we will not discuss this further here. The results from \cite{Combs} are not used in this paper.  
\end{remark} 

\subsection{Questions}\label{future}
We list here some questions for future research. 
\begin{enumerate}
\item  How are the families for different values of $k,l$ related in the
moduli space? Do they belong to distinct deformation classes?

\item What is the local Kuranishi dimension? Does varying \(u\), with the logarithmic transformation data fixed, give all sufficiently small deformations?

\item For fixed multiplicities and fixed $u$, do different choices of the admissible translation vectors produce even more biholomorphically distinct complex structures on $S^6$?

\item Other starting rational elliptic surfaces have been found to produce complex structures on $S^6$ \cite{Explorer, LR26, Via}, corresponding to many cases on Oguiso-Shioda's list with Mordell-Weil group $\ZZ$ \cite{OS}. 
Which other starting configurations admit infinitely many allowable multiplicities also yielding $S^6$?

\end{enumerate}

\subsection{Outline} 
This manuscript was originally intended to only be lecture notes to give an alternative presentation of some of the difficult computations in~\cite{EngelArxiv}. 
In Section~\ref{back}, we recall all the necessary details from~\cite{EngelArxiv},
and write down the monodromy matrices in a convenient form,
where one can see the original monodromy matrix of the rational elliptic surface as the central $2 \times 2$ block; this is similar to how Wang writes the monodromy matrices in \cite{Wang}. 
Note the monodromy matrices are independent of the logarithmic transformation data. Subsection~\ref{fillings} is where we introduce the logarithmic transformations. 
In Section~\ref{fund}, we compute the fundamental group using Van Kampen's Theorem. The fundamental group relies on this additional logarithmic transformation data, not only on the monodromy. For the computation of the integral homology groups in Section~\ref{hom}, we give a slightly different approach than in~\cite{EngelArxiv}. We show that, if the total space is simply-connected, then vanishing of the second integral homology group reduces to showing that a certain explicit $4 \times 4$ integral matrix, depending only upon the logarithmic transformation data, is unimodular.  Once the calculation was written in this way, the author then realized that $S^6$ also arises for infinitely many choices of the multiplicities, and the evolution of this manuscript from lecture notes to the present form was the result of that. 

\subsection{AI Statement} AI tools were used to help in drafting parts of the paper and performing some computational steps, under explicit directions from the author. All arguments and calculations have been verified by the author, who takes responsibility for the mathematical content.  

\subsection{Acknowledgements} The author would like to thank Philip Engel, Nobuhiro Honda, Wenfei Liu, and S\"onke Rollenske for helpful discussions. This research was partially supported by NSF Grant DMS-2404195.

\section{Background}
\label{back}

We briefly summarize the construction of Alp\"oge-Claude \cite{AC}, as explained by~\cite{EngelArxiv}.  Take a rational elliptic surface $\pi : S \rightarrow \PP^1$ with singular fibers of Kodaira types $\mathrm{IV}^*, \mathrm{III}, \mathrm{I}_1$ over the  points $0,1, \infty$, respectively.  This fibration has Mordell--Weil group $\ZZ$,  generated by a section $P$, and zero section $O$. Consider the line bundle $M = \mathcal{O}_S(P - O)$,
and let $M^{\circ}$ denote its restriction to $S \setminus \pi^{-1}(\infty)$. 
Let $0_M:S\to M$ denote the zero section of the line
bundle, and let $M^{\times,\circ} = M^\circ \setminus 0_M(S)$.
Choose a reference section $\phi\in\mathcal{H}om(t_{6P}^{*}M,M)(S)$ which is an isomorphism away from the fiber over $\infty$ and has a simple zero along that fiber, as in~\cite[Corollary~2.2]{EngelArxiv}. For $u\ne0$ with $|u|$ sufficiently small, the map $u\phi$ induces an automorphism of $M^{\times,\circ}$ covering translation by $-6P$ on $S\setminus\pi^{-1}(\infty)$. Taking the quotient by the cyclic group generated by this automorphism gives a fibration $X^{\circ}(u)\to\PP^1\setminus\{\infty\}$, with singular fibers over $0$ and $1$ and all other fibers smooth complex $2$-tori.

The fibration is filled in over $\infty$ by a Mumford rank $2$ hexagonal degeneration,
which adds a singular fiber with 3 double curves and 2 triple points, which normalizes to a degree $6$ del Pezzo surface $dP_6$. This produces a fibration $X(u) \rightarrow \PP^1$.
The fiber $X_0$ over $0$ has $7$ irreducible components consisting of two primary Hopf surfaces and five ruled surfaces over elliptic curves; there are $7$ components because the $\mathrm{IV}^*$ fiber consists of $7$ irreducible rational curves. The fiber $X_1$ over $1$ consists of two primary Hopf surfaces; there are $2$ components because the $\mathrm{III}$ fiber consists of $2$ irreducible rational curves. We will see that both of these fibers will be discarded and replaced with non-reduced fibers (with smooth reductions) in the logarithmic transformation step. 

Fix the parameter $u$ of Engel's construction, and put
$U=\PP^1\setminus\{0,1,\infty\}$ and $V=X(u)|_U$.
Write $f':V\to U$ for the restricted fibration.
We fix a section $\sigma:\PP^1\to X(u)$ as follows.
Since $M|_O \cong \mathcal O_{\PP^1}(1)$, choose a holomorphic section
$\eta:O\to M|_O$ with a simple zero over $\infty$ and no other zeros.
Away from $\infty$, $\eta$ chooses a nonzero vector above each point
$O(t)$, so its image lies in $M^{\times,\circ}$. Its quotient image
defines a section of $X^\circ(u)$, which extends to $\sigma$ over the
Mumford filling by~\cite[Proposition~2.9]{EngelArxiv}. 
Fix a regular value $t$ and identify
$\Lambda=H_1(X(u)_t;\ZZ)\cong\ZZ^4$.
Choose loops $\gamma_0,\gamma_1,\gamma_\infty$ based at $t$,
each following a path from $t$ to a small counterclockwise circle around the corresponding  point and returning along the same path to $t$, with $\gamma_0\gamma_1\gamma_\infty=1$. 
Matrices act on column vectors, and in a product of paths the rightmost path is traversed first. With these conventions, $\gamma_0\gamma_1\gamma_\infty=1$ gives $A_0A_1A_\infty=I$.
\begin{lemma}
\label{mono}
There is an integral basis $(\delta,a,b,d)$ of $\Lambda$ in which the
homological monodromies are given by 
\begin{align}
\label{amat}
 A_0=\begin{pmatrix}
 1&-1&0&2\\0&-1&-1&6\\0&1&0&0\\0&0&0&1
 \end{pmatrix},
 A_1=\begin{pmatrix}
 1&0&1&-3\\0&0&1&0\\0&-1&0&6\\0&0&0&1
 \end{pmatrix},
 A_\infty=\begin{pmatrix}
 1&0&0&1\\0&1&1&0\\0&0&1&0\\0&0&0&1
 \end{pmatrix}.
\end{align}
\end{lemma}
\begin{proof}
We use the monodromy matrices in~\cite[Theorem~5.1]{EngelArxiv}. Let $(e_1,e_2,e_3,e_4)$ be Engel's basis, and define
\begin{align}
 \delta=-2e_1-e_2+3e_3,\qquad
 a=e_1-e_3,\qquad b=e_3,\qquad d=e_4-2e_1.
 \label{basis}
\end{align}
This change of basis has determinant one. Substitution in Engel's
matrices, with $T_1$ written as in his equation~(6.1), gives
\eqref{amat}.
\end{proof}
\begin{remark}The class $\delta$ represents the positively oriented circle in the original $\CC^\times$ fiber. The matrices show that the globally monodromy-invariant sublattice of $\Lambda$ is exactly $\ZZ\delta$.
\end{remark}
Consider the dual basis $(\delta^{\vee}, a^{\vee}, b^{\vee}, \psi)$ of $\Lambda^{\vee}$, where $\psi=d^\vee$. The matrices give the following:
\begin{align}
\label{invar1}
 H^0(U,R^1f'_*\ZZ)&=\ZZ\psi,  \quad \psi=d^\vee,\\
\label{invar2}
 H^0(U,R^2f'_*\ZZ)&=\ZZ\xi ,  \quad \xi=a^\vee\wedge b^\vee+6d^\vee\wedge\delta^\vee.
\end{align}
These agree with~\cite[Lemmas~3.8 and~7.3]{EngelArxiv}.
For \eqref{invar1}, note that an invariant covector
annihilates $(A_\infty-I)\Lambda=\ZZ\delta+\ZZ a$; invariance under
$A_1$, which sends $a$ to $-b$, then forces it to annihilate $b$.
The form $\xi$ in \eqref{invar2} is monodromy invariant. 
For $i \in \{0,1, \infty\}$, denote
\begin{align}
\Lambda^{A_i} := \{ \lambda \in \Lambda \ | \ A_i \lambda = \lambda \} = \Ker(A_i - I). 
\end{align}
Defining the elements
\begin{align}
u_0=d+2a+2b, \quad u_1=d+3a+3b,
 \label{uvect}
\end{align}
we have 
\begin{align}
 \Lambda^{A_i} =\ZZ\delta\oplus\ZZ u_i\quad(i=0,1).
\end{align}
We also have 
\begin{align}
K_\infty:= \Lambda^{A_{\infty}} =(A_\infty-I)\Lambda =\ZZ\delta\oplus\ZZ a.
 \label{fixlat}
\end{align}

\subsection{The fillings and logarithmic transformations}
\label{fillings}

For notational simplicity, we denote the total space after the logarithmic
transformations by $X$, and write $f:X\to\PP^1$ for the completed
fibration. Choose disjoint discs
$\Delta_i$, closed subdiscs $D_i$, and their positive boundary circles
$S_i^1=\partial D_i$, and set
\begin{align}
 V_i=f^{-1}(\Delta_i),\qquad N_i=f^{-1}(D_i),\qquad
 M_i^5=\partial N_i=f^{-1}(S_i^1),\qquad F_i=f^{-1}(i)_{\red},
 \label{opens}
\end{align}
for $i \in \{0,1, \infty\}$.
Then $N_i\simeq V_i\simeq F_i$ by the local deformation retraction
\cite[Proposition~3.1]{HVFibrationsPublished}, and $V\cap V_i\simeq M_i^5$.
The Mumford fiber $F_\infty$ has normalization $P=dP_6$;
opposite sides of its boundary hexagon $D$ are identified
~\cite[Proposition~2.8]{EngelArxiv}.
In the covering $(\CC^\times)^2$ of a nearby torus, the group
$S^1 \times S^1$ acts by rotating the two coordinates:
\begin{align}
 (e^{2\pi\sqrt{-1}\theta_1},e^{2\pi\sqrt{-1}\theta_2})\cdot(z,w)
 =(e^{2\pi\sqrt{-1}\theta_1}z,e^{2\pi\sqrt{-1}\theta_2}w).
 \label{rotate}
\end{align}
These rotations extend over the toric boundary and commute with the
period identifications, so they descend to the Mumford filling.
Their circle classes generate $K_\infty=\ZZ\delta\oplus\ZZ a$;
in particular, $\delta$ and $a$ determine circle actions on the filling.

For $i = 0,1$, we use the good reduction and logarithmic transformation construction of~\cite[Section~3]{EngelArxiv}, but allow higher multiples of the base-change
degrees. Let
\begin{align}
\label{divcon}
 m_0=3k_0,\qquad m_1=4\ell_0,
\end{align}
where $k_0,\ell_0$ are positive integers, choose 
\begin{align}
 a_i\in\tfrac1{m_i}\Lambda^{A_i},\qquad  v_i=m_i a_i=r_i\delta+n_i u_i,
 \label{logdat}
\end{align}
where $r_i,n_i\in\ZZ$. We refer to the $4$-tuple $(m_0,a_0,m_1,a_1)$ as the \textit{logarithmic transformation data}.
We will only consider logarithmic transformation data satisfying $\gcd(n_i,m_i)=1$; this condition will ensure that the cyclic actions below are free.

Let $d_0=3$ and $d_1=4$. For a local coordinate $q$ at $i$,
Engel's good-reduction model is a smooth torus family over the cover
$q=z^{d_i}$. Pull this family back by the further base change
$z=s^{m_i/d_i}$. Since smoothness is preserved under base change,
this gives a smooth torus family $\mathcal T_i\to\widetilde\Delta_i$
over the degree-$m_i$ cover $q=s^{m_i}$. Let $\zeta_i=e^{2\pi\sqrt{-1}/m_i}$.
Over $s\ne0$, the base-changed family consists of pairs $(p,s)$ with
$(q\circ f')(p)=s^{m_i}$. The cyclic symmetry
$(p,s)\mapsto(p,\zeta_i s)$ fixes $p$ and permutes its $m_i$
preimages under the covering. The degree-$d_i$ symmetry extends
to Engel's smooth model \cite[Corollary~3.3 and Construction~3.5]{EngelArxiv}.
Combining it with $s\mapsto\zeta_i s$ gives the extension to
$\mathcal T_i$, since $(\zeta_i s)^{m_i/d_i}=e^{2\pi\sqrt{-1}/d_i}s^{m_i/d_i}$.

Choose a local holomorphic section
$o_i:\widetilde\Delta_i\to\mathcal T_i$, defined on the whole covering
disc, and use $o_i(s)$ as the origin of $\mathcal T_{i,s}$.  
For $x \in \mathcal{T}_{i,s}$, the extended symmetry then has the form $(x,s)\mapsto(B_sx+b_s,\zeta_i s)$, where $B_s:\mathcal T_{i,s}\to\mathcal T_{i,\zeta_i s}$ is a group isomorphism and $b_s\in\mathcal T_{i,\zeta_i s}$.
Note that $B_s$ is defined on the whole disc. Consider 
\begin{align}
 \rho_i(x,s)=(B_sx,\zeta_i s).
 \label{cyclic}
\end{align}
Thus $\rho_i$ sends $o_i(s)$ to $o_i(\zeta_i s)$
and $\rho_i^{m_i}=1$.
Over $s\ne0$, let $e_i(s)$ denote the pullback of $\sigma$, viewed as a point of $\mathcal{T}_{i,s}$ with origin $o_i(s)$. Since $\sigma$ is a
section downstairs, the original cyclic symmetry sends $e_i(s)$ to
$e_i(\zeta_i s)$, and hence
\begin{align*}
 e_i(\zeta_i s)=B_se_i(s)+b_s.
\end{align*}
The translation $x\mapsto x-e_i(s)$ sends this pulled-back section to
$o_i(s)$ and conjugates the original punctured action to
\eqref{cyclic}, since
\begin{align*}
 (B_sx+b_s)-e_i(\zeta_i s)=B_s(x-e_i(s)).
\end{align*}

To define the translation used in the logarithmic transformation, write each fiber as
\begin{align}
 \mathcal T_{i,s}=E_s/\Lambda_s,\qquad
 E_s=T_{o_i(s)}\mathcal T_{i,s}\cong\CC^2,
 \label{tangent}
\end{align}
where $\Lambda_s$ is its period lattice in the tangent space at the
origin. The spaces $E_s$ form a holomorphic vector bundle; we also
write $B_s$ for the induced linear map $E_s\to E_{\zeta_i s}$.
Parallel transport along the chosen paths identifies $\Lambda$ with $\Lambda_s$; write $v_i(s)$
for the period corresponding to $v_i$, and
$a_i(s)=[v_i(s)/m_i]$ for the corresponding point of the torus.
Under these flat lattice identifications, $B_s$ acts on the period lattices by $A_i^{-1}$, as in~\cite[Lemma~3.6]{EngelArxiv}. In particular, $B_s v_i(s)=v_i(\zeta_i s)$, and $B_s$ preserves the covector $\psi$. Let $t_{a_i}$ be fiberwise translation by $a_i(s)$, using the origin $o_i(s)$, defined by
\begin{align}
t_{a_i}(x,s)=(x+a_i(s),s).
\end{align}
\begin{proposition}
\label{logprop}
For the data in \eqref{logdat}, with $\gcd(n_i,m_i)=1$, the cyclic
action generated by $\widetilde\rho_i=t_{a_i}\rho_i$ has order
$m_i$ and is free. Its quotient is smooth, with central fiber
$m_iF_i$, where $F_i$ is bielliptic. Over the punctured disc, using the identification $x\mapsto x-e_i(s)$ above, a biholomorphism from the original family
to this quotient is induced by
\begin{align}
 \Psi_i(x,s)=(x+\sigma_i(s),s), \quad
 \sigma_i(s)=\left[\frac{\log s}{2\pi\sqrt{-1}}v_i(s)\right]
                 \in E_s/\Lambda_s.
 \label{glue}
\end{align}
\end{proposition}

\begin{proof}
Since $v_i$ is invariant, $B_sv_i(s)=v_i(\zeta_i s)$, and
$t_{a_i}$ commutes with $\rho_i$. Thus
$\widetilde\rho_i^{m_i}=t_{v_i}=1$; the action on the base has
order $m_i$. Away from $s=0$, no nonidentity power fixes a point
because it moves $s$. On the central real torus
$(\Lambda\otimes\RR)/\Lambda$, the map to $\RR/\ZZ$ induced by
$\psi$ changes under $\widetilde\rho_i$ by $n_i/m_i$.
Since $\gcd(n_i,m_i) = 1$, no nonidentity power fixes any point on the central fiber.  
Therefore the quotient is smooth, and $q=s^{m_i}$ gives multiplicity
$m_i$ to its central fiber. The central fiber of $\mathcal T_i$
is an abelian surface, and the linear action has a one-dimensional
fixed complex subspace and a nontrivial eigenvalue of order three
or four. Its free quotient $F_i$ is therefore bielliptic by the same argument as in
\cite[Corollaries~3.2 and~3.9]{EngelArxiv} (although the precise bielliptic surfaces might be different; we do not need to identify them explicitly here). 

Changing the branch of $\log s$ adds an integral multiple of
$v_i(s)$, so \eqref{glue} defines a single-valued holomorphic
translation of the punctured torus family. Moreover,
\begin{align}
 \sigma_i(\zeta_i s)=B_s\sigma_i(s)+a_i(\zeta_i s), \quad
 \Psi_i\circ\rho_i=\widetilde\rho_i\circ\Psi_i.
 \label{descend}
\end{align}
Thus $\Psi_i$ carries orbits of $\rho_i$ to orbits of
$\widetilde\rho_i$. It descends to the claimed biholomorphism;
translation by $-\sigma_i$ induces its inverse.
\end{proof}

We use the biholomorphism in Proposition~\ref{logprop} to glue the quotient to $V$,
to obtain a threefold $X_{m_0, a_0, m_1, a_1}(u) \rightarrow \PP^1$. 
The chosen gluing identification depends on the actual vector $a_i$:
adding an invariant integral vector leaves $t_{a_i}$ unchanged but
can change \eqref{glue}.
For the later boundary calculations, normalize $|q|=1$ on $S_i^1$
and write $q=e^{2\pi\sqrt{-1}\theta}$,
$s=e^{2\pi\sqrt{-1}\theta/m_i}$. In flat real torus coordinates,
\eqref{glue} is then translation by $\theta a_i$.

\section{Fundamental group}
\label{fund}
In the following we will abbreviate $X_{m_0, a_0, m_1, a_1}(u)$ by $X$. 
Use the section $\sigma|_U$ and the
chosen paths to identify all fiber lattices with $\Lambda$, and
write $x_i$ for the lift $\sigma\circ\gamma_i$ of $\gamma_i$.
Note that $\langle x_0,x_1\rangle$ is free, $x_0x_1x_\infty=1$, and
$x_i\lambda x_i^{-1}=A_i\lambda$, for any $\lambda \in \Lambda$. 
Since $U$ retracts onto a wedge of two circles, $\pi_2(U)=0$.
The homotopy sequence and the section $\sigma|_U$ therefore give
\begin{align}
\label{semi}
 1\longrightarrow\Lambda\longrightarrow\pi_1(V)
 \longrightarrow\pi_1(U)\longrightarrow1,\qquad
 \pi_1(V)=\Lambda\rtimes\pi_1(U).
\end{align}
Similarly, the bundle $T^4\to M_i^5\to S_i^1$ and its restricted section
$\sigma|_{S_i^1}$ give
\begin{align}
 1\longrightarrow\Lambda\longrightarrow\pi_1(M_i^5)
 \longrightarrow\ZZ\longrightarrow1,\qquad
 \pi_1(M_i^5)=\Lambda\rtimes_{A_i}\ZZ.
\label{u8}
\end{align}
The generator of the last $\ZZ$ is $x_i$.
Thus the inclusion into the smooth part is explicitly
\begin{align}
\label{incl}
 \iota_i:\pi_1(M_i^5)\longrightarrow\pi_1(V),\qquad
 \iota_i(\lambda,k)=\lambda x_i^{k}.
\end{align}
The inclusion $M_i^5 \hookrightarrow V_i$ induces the homomorphism 
\begin{align}
 j_i:\pi_1(M_i^5)\longrightarrow\pi_1(V_i)\cong\pi_1(F_i).
\label{u10}
\end{align}
The open cover $X=V\cup V_0\cup V_1\cup V_\infty$ has connected
intersections $V\cap V_i$, and the $V_i$ are pairwise disjoint.
Van Kampen's theorem gives a presentation of $\pi_1(X)$ as follows.
Take the free product $\pi_1(V) * \pi_1(F_0) *  \pi_1(F_1) * \pi_1(F_{\infty})$,
and impose
\begin{align}
 \iota_i(h)=j_i(h),\qquad
 h\in\pi_1(M_i^5),\quad i=0,1,\infty;
\label{u11}
\end{align}
see \cite[Section~1.2]{Hatcher}.

For $i=0,1$, let $\mu_i\in\pi_1(M_i^5)$ be represented by the boundary of a small complex disc transverse to $F_i$, oriented counterclockwise in that disc and viewed as a loop in $M_i^5$. We call this the \textit{positive normal circle}.
\begin{proposition}
\label{bound}
The three maps $j_i$ are surjective for  $i = 0,1,\infty$.
At infinity, we have
\begin{align}
 \pi_1(F_\infty)=\Lambda/K_\infty,\qquad
 j_\infty(\lambda,k)=[\lambda].
 \label{jinf}
\end{align}
At $i=0,1$, write $g_i=j_i(x_i)$. Then
\begin{align}
 \pi_1(F_i)&=\langle\Lambda,g_i\mid[\Lambda,\Lambda]=1,\ 
 g_i\lambda g_i^{-1}=A_i\lambda,\ g_i^{m_i}=v_i\rangle,
 \label{pres}\\
 j_i(\lambda,k)&=\lambda g_i^k,\qquad
 \Ker(j_i)=\ZZ\langle\mu_i\rangle,\qquad
 \mu_i=(-v_i,m_i).
 \label{jfin}
\end{align}
\end{proposition}

\begin{proof}
Before the period-lattice quotient, the Mumford central fiber is a
union of simply connected toric surfaces with simply connected
intersections and dual complex a triangulated plane. Van Kampen makes
this union simply connected. The group of covering transformations is $\Lambda/K_\infty$:
the loops in $K_\infty$ become contractible in the toric surfaces.
The same covering extends over the filling, so the map from a nearby
fiber induces $\lambda\mapsto[\lambda]$ in \eqref{jinf}.
The loop $\sigma(S_\infty^1)$ bounds the section disc
$\sigma(D_\infty)$. The loop $x_\infty$ is obtained from this boundary
loop using the chosen path to the base point. Hence $j_\infty(x_\infty)=1$.
Together these two observations give
$j_\infty(\lambda,k)=[\lambda]$ for every $(\lambda,k)$.

For $i = 0$ or $i = 1$, we can choose a smooth trivialization of $\mathcal T_i$ over
the covering disc which respects the finite cyclic action. Such a trivialization is obtained by averaging a connection over the finite cyclic group and using parallel transport along radial paths. Its quotient identifies $N_i$ with a disc bundle in the complex normal
line bundle of $F_i$, and $M_i^5$ with its boundary circle bundle.
Since $F_i$ has a torus cover, $\pi_2(F_i)=0$. The homotopy sequence
of this circle bundle therefore gives a surjection
$\pi_1(M_i^5)\to\pi_1(F_i)$ with kernel an infinite cyclic group
generated by the positive normal circle.

To express the positive normal circle using our generators in $V$, choose a
transverse coordinate $w$ with $q=w^{m_i}$. One positive turn of
$w$ makes $m_i$ turns in the base, while a lift of $\sigma_i$ in
\eqref{glue} to $E_s$ increases by $v_i$.
The inverse gluing map subtracts $\sigma_i$, so the fiber displacement
in the coordinates of $V$ is $-v_i$.
The positive normal circle therefore represents $(-v_i,m_i)=v_i^{-1}x_i^{m_i}$.
Quotienting by this kernel imposes $x_i^{m_i}=v_i$, proving
\eqref{pres}--\eqref{jfin}.
\end{proof}

\begin{theorem}
\label{simply}
The logarithmic transformation data satisfying
\begin{align}
 \left|m_1n_0+m_0n_1\right|
 =\left|m_0m_1\psi(a_0+a_1)\right|=1
 \label{bezout}
\end{align}
give a simply connected $X$, where we also write $\psi$ for its $\QQ$-linear extension to $\Lambda\otimes\QQ$. Such data exist precisely when
$\gcd(m_0,m_1)=1$. In particular, every relatively prime pair
$3k_0,4\ell_0$ admits such choices.
\end{theorem}
\begin{proof}
We use multiplicative notation for the fiber group $\Lambda$, whose generators $\delta,a,b,d$ commute pairwise. The base group has the presentation $\pi_1(U)=\langle x_0,x_1,x_\infty\mid x_0x_1x_\infty=1\rangle$. We use the same symbols for their lifts via $\sigma$ to
$\pi_1(V)$. An element of $\pi_1(V)$ can be written as $\lambda w$, where $\lambda\in\Lambda$ and $w$ is a word in $x_0,x_1$ and their inverses; the third generator is $x_\infty=(x_0x_1)^{-1}$. The semidirect product in \eqref{semi} means that conjugation by $x_i$ acts on $\Lambda$ by $A_i$. Consequently, \eqref{amat} gives the following presentation of $\pi_1(V)$: the generators are $\delta,a,b,d,x_0,x_1,x_\infty$, the generators $\delta,a,b,d$ commute pairwise, and the remaining defining relations are
\begin{align}
\label{vpres1}
x_0x_1x_\infty&=1, &  x_i\delta x_i^{-1} &=\delta\quad(i=0,1,\infty), &  & \\
\label{vpres2}
x_0ax_0^{-1}&=\delta^{-1}a^{-1}b, &  x_0bx_0^{-1}&=a^{-1}, & x_0dx_0^{-1}&=\delta^2a^6d, \\
\label{vpres3}
x_1ax_1^{-1}&=b^{-1}, & x_1bx_1^{-1}&=\delta a, &  x_1dx_1^{-1}& =\delta^{-3}b^6d, \\
x_\infty ax_\infty^{-1}&=a, & x_\infty bx_\infty^{-1}&=ab, &  x_\infty dx_\infty^{-1}&=\delta d.
\label{vpres4}
\end{align}
To obtain $\pi_1(X)$, we start with the free product $\pi_1(V) * \pi_1(F_0) * \pi_1(F_1)*\pi_1(F_\infty)$ and impose the Van Kampen relations \eqref{u11}: for each $i$, we identify $\iota_i(h)$ with $j_i(h)$ for every $h\in\pi_1(M_i^5)$. Proposition~\ref{bound} states that each $j_i$ is surjective, so these identifications allow every generator from a filling group $\pi_1(F_i)$ to be replaced by an element coming from $\pi_1(V)$. We now carry out these replacements. We use the same symbols for the generators of $\pi_1(V)$ and their images in $\pi_1(X)$.

At infinity, the Van Kampen relation \eqref{u11} identifies $\iota_\infty(\lambda,k)$ with $j_\infty(\lambda,k)$. By \eqref{incl}, the first is $\lambda x_\infty^k\in\pi_1(V)$, and by \eqref{jinf}, the second is $[\lambda]\in\Lambda/K_\infty=\pi_1(F_\infty)$. Taking $k=0$ therefore identifies every class $[\lambda]$ in the infinity factor with the image of $\lambda$ in $\pi_1(X)$. Since $K_\infty=\ZZ\delta\oplus\ZZ a$ by \eqref{fixlat}, the classes of $\delta$ and $a$ in $\pi_1(F_\infty)$ are the identity. Their images in $\pi_1(X)$ consequently satisfy $\delta=a=1$. Taking $\lambda=0$ and $k=1$ also gives $x_\infty=1$ in $\pi_1(X)$. We can therefore remove the infinity factor from the presentation by adding the relations $\delta=a=x_\infty=1$ to the presentation of $\pi_1(V)$.

For $i=0,1$, we again apply the Van Kampen relation \eqref{u11}, using the formulas for $\iota_i$ and $j_i$ in \eqref{incl} and \eqref{jfin}. Taking $h=(\lambda,0)$ identifies the lattice element $\lambda$ in the presentation \eqref{pres} of $\pi_1(F_i)$ with the corresponding element of $\pi_1(V)$. Taking $h=(0,1)$ identifies $g_i$ with $x_i$. We may thus replace the lattice generators and $g_i$ in \eqref{pres} by the corresponding generators of $\pi_1(V)$. After these replacements, the commutativity and conjugation relations in \eqref{pres} are already among the defining relations of $\pi_1(V)$ listed above. The remaining relation $g_i^{m_i}=v_i$ becomes $x_i^{m_i}=v_i$ in $\pi_1(X)$. Thus $\pi_1(X)$ has the generators and relations of $\pi_1(V)$, together with
\begin{align}
\delta=a=x_\infty=1,\qquad x_0^{m_0}=v_0,\qquad x_1^{m_1}=v_1.
\label{fillrel}
\end{align}

We next simplify the above presentation. Since $a=1$, the relation $x_1ax_1^{-1}=b^{-1}$ in \eqref{vpres3} gives $b=1$ in $\pi_1(X)$. Also, $x_\infty=1$ and $x_0x_1x_\infty=1$ give $x_1=x_0^{-1}$. Thus $\pi_1(X)$ is generated by the images of $d$ and $x_0$, which we denote by $c$ and $x$, respectively. After these substitutions, the relations in \eqref{vpres1}-\eqref{vpres4} reduce to $cx=xc$. 

We now use additive notation. By \eqref{uvect}, both $u_0$ and $u_1$ map to $c$, so \eqref{logdat} shows that $v_i$ maps to $n_i c$. Since $x_0$ maps to $x$ and $x_1$ maps to $-x$, the two power relations in \eqref{fillrel} become
\begin{align}
-n_0c+m_0x=0,\qquad -n_1c-m_1x=0.
\label{relat}
\end{align}
The coefficient matrix of the relations in \eqref{relat} has determinant $m_1n_0+m_0n_1=\pm1$ by \eqref{bezout}. Thus these relations give $c=x=0$, so $\pi_1(X)=1$.

Finally, if $\gcd(m_0,m_1)=1$, there exist integers $n_0,n_1$ satisfying $m_1n_0+m_0n_1=1$. This equation also implies $\gcd(n_i,m_i)=1$ for both indices. Choosing arbitrary integers $r_i$ and setting $a_i=(r_i\delta+n_i u_i)/m_i$ therefore gives admissible data in \eqref{logdat} satisfying \eqref{bezout}. Conversely, every common divisor of $m_0,m_1$ divides $m_1n_0+m_0n_1$, so \eqref{bezout} cannot hold unless $\gcd(m_0,m_1)=1$.

\end{proof}

\section{Integral homology}
\label{hom}
Assume \eqref{bezout}, so $\pi_1(X)=H_1(X;\ZZ)=0$.
All homology and cohomology groups below have integral coefficients.

\subsection{The Mayer--Vietoris reduction}
Since the $V_i$ are disjoint, apply the ordinary Mayer--Vietoris
sequence to $V$ and $V_0\sqcup V_1\sqcup V_\infty$:
\begin{equation}
\begin{tikzcd}[column sep=2.2em,row sep=0em]
\textstyle\bigoplus\limits_i H_2(M_i^5)
 \arrow[r,"\Phi_2"] &
H_2(V)\oplus\textstyle\bigoplus\limits_i H_2(F_i)
 \arrow[r] & H_2(X)
 \arrow[dll,to path={
 (\tikztostart.east)
 .. controls +(2em,0) and +(2em,0) ..
 ($(\tikztostart.east)!0.5!(\tikztostart.east |- \tikztotarget.west)$)
 --
 ($(\tikztotarget.west)!0.5!(\tikztotarget.west |- \tikztostart.east)$)
 .. controls +(-2em,0) and +(-2em,0) ..
 (\tikztotarget.west)\tikztonodes}] \\
\textstyle\bigoplus\limits_i H_1(M_i^5)
 \arrow[r,"\Phi_1"] &
H_1(V)\oplus\textstyle\bigoplus\limits_i H_1(F_i)
 \arrow[r] & H_1(X)=0.
\end{tikzcd}
\label{mvseq}
\end{equation}
The maps are $\Phi_j((z_i))=(\sum_i(\iota_i)_*z_i,-((j_i)_*z_i)_i)$.
For the compact six-manifolds $N_i$ with boundaries $M_i^5$ in \eqref{opens}, define
$\tau_i$ to be the following composition:
\begin{align}
 \tau_i:H^3(F_i)\cong H_3(N_i,M_i^5)
 \xrightarrow{\partial_i}H_2(M_i^5)
 \xrightarrow{(\iota_i)_*}H_2(V),
 \label{taudef}
\end{align}
using the retraction and Poincar\'e--Lefschetz duality.
For $i=0,1$, we choose the sign of this duality isomorphism
so that, under the Poincar\'e duality identification
$H^3(F_i)\cong H_1(F_i)$, the boundary associated to an oriented loop
is the torus formed by its positive normal circles, oriented by the
loop first and the circle second.
\begin{proposition}
\label{mvred}
There is an isomorphism
\begin{align}
 H_2(X)\cong\Coker\left(\textstyle\bigoplus\limits_iH^3(F_i)
             \xrightarrow{\tau_0+ \tau_1+ \tau_{\infty}}H_2(V)\right).
 \label{coker}
\end{align}
\end{proposition}
\begin{proof}
For automorphisms $A_j$ of a lattice $L$, the quotient
$L/\sum_j(A_j-I)L$ is called the group of coinvariants. When the
$A_j$ generate the monodromy of a local system, this quotient
identifies with its degree-zero homology; see
\cite[Proposition~5.14(1)]{DavisKirk}.

Abelianizing \eqref{semi} and using \eqref{amat} gives
$H_1(V)=\ZZ^3$ and $H_1(M_i^5)=\ZZ^3$.
Indeed, the individual coinvariants of $A_0,A_1$ have respectively
$b=-a,\delta=-3a$ and $b=-a,\delta=-2a$, leaving $a,d$ free.
Their common coinvariants are $\ZZ[d]$.
Each $F_i$ has first Betti number two. Since $H_1(X)=0$,
$\Phi_1$ is a surjection from $\ZZ^9$ to a rank-nine group;
its kernel, a rank-zero subgroup of $\ZZ^9$, is zero.
Thus $H_2(X)=\Coker(\Phi_2)$, and each $H_1(F_i)$ is free.

Next, $H^4(F_i)=\ZZ$, evaluated on $[F_i]$. For $i = 0,1$ this is obvious since $F_i$ is a manifold. For $i=\infty$, use the standard top-degree
calculation for an irreducible compact complex surface; see the discussion in \cite[p.~30]{HVSixSpherePublished}. The pair sequence and duality give
\begin{align}
 H_2(M_i^5)\longrightarrow H_2(N_i)
 \longrightarrow H_2(N_i,M_i^5)\cong H^4(F_i) \cong \ZZ.
 \label{pair}
\end{align}
The image of $z\in H_2(N_i)$ evaluates on $[F_i]$ as the ambient
intersection $z\cdot F_i$.
The divisor $m_iF_i$ is principal, where $m_\infty=1$; hence
$m_i(z\cdot F_i)=0$, so the middle map is zero. Thus
$H_2(M_i^5)\to H_2(F_i)$ is surjective for every $i$.

Consequently every class in $\Coker(\Phi_2)$ has a representative in
$H_2(V)$ alone. Its relations are precisely the images of
$\Ker(H_2(M_i^5)\to H_2(F_i))$. The pair sequence identifies these
kernels with $\Image\partial_i$, proving \eqref{coker}.
\end{proof}

\subsection{Homology of the smooth part}
We next compute the second integral homology of the smooth part $V$. 
\begin{lemma}
\label{vbasis}
We have $H_2(V)\cong\ZZ^6$.
\end{lemma}
\begin{proof}
In local-coefficient homology on $U$, we use $\Lambda$ also to denote
the local system with fiber $H_1(X(u)_t;\ZZ)$ and monodromies $A_i$.
The homology Serre spectral sequence of $T^4\to V\to U$ has
\begin{align}
 E^2_{p,q}=H_p(U;\textstyle\bigwedge^q\Lambda)\Longrightarrow H_{p+q}(V).
\label{serre}
\end{align}
Here $H_q(T^4;\ZZ)=\bigwedge^q\Lambda$, with the local coefficient
system induced by monodromy; see \cite[Theorem~9.6]{DavisKirk}.
Since $U$ retracts onto a wedge of two circles, all columns with
$p\geq2$ vanish. There are therefore no differentials, and the
filtration in total degree two gives a short exact sequence
\begin{align}
 0\longrightarrow H_0(U;\textstyle\bigwedge^2\Lambda)
 \longrightarrow H_2(V)
 \xrightarrow{p}H_1(U;\Lambda)\longrightarrow0.
\label{graph}
\end{align}
We compute both end terms over $\ZZ$.

For the left term in \eqref{graph}, the coinvariant description recalled above gives
\begin{align}
 H_0(U;\textstyle\bigwedge^2\Lambda)
 =\frac{\bigwedge^2\Lambda}
 {(\bigwedge^2A_0-I)\bigwedge^2\Lambda+
  (\bigwedge^2A_1-I)\bigwedge^2\Lambda}.
\label{coinv}
\end{align}
Use the exterior basis
\begin{align}
 s_1&=\delta\wedge a,\quad s_2=\delta\wedge b,\quad
 s_3=\delta\wedge d,\label{wedgeone}\\
 s_4&=a\wedge b,\quad s_5=a\wedge d,\quad s_6=b\wedge d.
 \label{wedgetwo}
\end{align}
The matrices \eqref{amat} give
\begin{align}
 (\textstyle\bigwedge^2A_0-I)s_4=s_1,\qquad
 (\textstyle\bigwedge^2A_1-I)s_4=s_2.
\label{wedgefirst}
\end{align}
Thus $s_1,s_2$ vanish in the quotient. Modulo these two classes,
the relations obtained from $s_5,s_6$ include
\begin{align}
 s_5+s_6=0,\qquad
 s_3+6s_4+2s_5-s_6=0,\qquad
 s_3+6s_4+s_5-s_6=0.
\label{exterior}
\end{align}
Subtracting the last two relations gives $s_5=0$, and then
$s_6=0$ and $s_3=-6s_4$. Hence the quotient is generated by $[s_4]$.
There is no further relation on this generator: the monodromy-invariant
alternating form \eqref{invar2}, written in the ordinary dual basis, is
\begin{align}
 \xi=a^\vee\wedge b^\vee+6d^\vee\wedge\delta^\vee,
 \qquad \langle\xi,s_4\rangle=1.
\label{wedgepair}
\end{align}
It descends to an integer-valued homomorphism on the quotient, proving
\begin{align}
 H_0(U;\textstyle\bigwedge^2\Lambda)=\ZZ[a\wedge b].
\label{fibercoinv}
\end{align}

For the right term of \eqref{graph}, cellular homology on the wedge of
two circles gives
\begin{align}
 H_1(U;\Lambda)=\Ker(\mathcal D_1),\label{edgekernel}
\end{align}
where 
\begin{align}
 \mathcal D_1:\Lambda\oplus\Lambda\longrightarrow\Lambda,
 \qquad\mathcal D_1(v_0,v_1)&=(A_0-I)v_0+(A_1-I)v_1.
 \label{edgemap}
\end{align}
There are no cellular two-chains on this graph. The matrices \eqref{amat}, or the common coinvariants computed in
Proposition~\ref{mvred}, show that the image of $\mathcal D_1$ is
$\ZZ\delta\oplus\ZZ a\oplus\ZZ b$. Thus its kernel is a subgroup
of $\ZZ^8$ of rank five, and is therefore free of rank five.
Sequence \eqref{graph} becomes
\begin{align}
 0\longrightarrow\ZZ\longrightarrow H_2(V)
 \xrightarrow{p}\ZZ^5\longrightarrow0.
\label{ranksplit}
\end{align}
The right term is free, so the sequence splits (though not canonically).  This finishes the proof. 
\end{proof}
We denote the generator of the left term in \eqref{ranksplit} by 
\begin{align}
\label{sdef}
s=[a\wedge b],
\end{align}
and also write $s$ for its image in $H_2(V)$.

\subsection{Removing the Mumford relations}
The section $\sigma|_U$ identifies the underlying real torus bundle with
$(\Lambda\otimes\RR)/\Lambda$ in flat local coordinates; only the
complex structure varies. For $\lambda\in\Lambda^{A_i}$, denote
by $T_{\lambda,i}$ the two-torus obtained by sweeping the loop
$\lambda$ along the loop $\gamma_i$ lifted using the section
$\sigma|_U$.
Orient it by the fiber loop first and the positive base circle second.
The map $p$ in \eqref{graph} gives the fiber one-cycles obtained
by cutting a two-cycle over the two chosen loops in the base.
Thus $p(T_{\lambda,0})=(\lambda,0)$ and
$p(T_{\lambda,1})=(0,\lambda)$.

\begin{proposition}
\label{mumford}
For $i = 0,1,\infty$, $H^3(F_i) \cong \ZZ^2$. The image of $\tau_\infty$ is the
primitive subgroup generated by $T_{\delta,\infty},T_{a,\infty}$,
and the quotient
\begin{align}
 Q_\infty:=H_2(V)/\Image(\tau_\infty)
 =\ZZ\langle s,t,t_0,t_1\rangle
\label{quotient}
\end{align}
is free of rank four, where $s$ is defined in \eqref{sdef}, and 
\begin{align}
(t,t_0,t_1)=([T_{\delta,0}],[T_{u_0,0}],[T_{u_1,1}]).
\end{align}
 Consequently
\begin{align}
 H_2(X)\cong\Coker\left(
 H^3(F_0)\oplus H^3(F_1)
 \xrightarrow{\ \overline\tau_0+\overline\tau_1\ }Q_\infty\right),
 \label{fourmap}
\end{align}
where $\overline\tau_i$ is $\tau_i$ from \eqref{taudef}, followed by the quotient map.
\end{proposition}
\begin{proof}
For $i=0,1$, Poincar\'e duality gives
$H^3(F_i)=H_1(F_i)=\ZZ^2$, including freeness by the proof of
Proposition~\ref{mvred}. At infinity the normalization pairs give
\begin{align}
 H^2(P)\oplus H^2(C)\longrightarrow H^2(D)
 \longrightarrow H^3(F_\infty)\longrightarrow0,
 \label{normal}
\end{align}
where $\nu:P\to F_\infty$ is the normalization and $C=\nu(D)$.
Write $P=\operatorname{Bl}_3\PP^2$, with divisor classes $H,E_1,E_2,E_3$.
The opposite pairs in $D$ are $(E_j,H-E_k-E_l)$, where
$\{j,k,l\}=\{1,2,3\}$. Quotienting $H^2(D)=\ZZ^6$ by the
pullbacks from $H^2(C)=\ZZ^3$ leaves the three differences of
degrees. The restriction of $AH-\sum c_jE_j$ has differences
$(c_1+c_2+c_3-A)(1,1,1)$. Therefore
\begin{align}
 H^3(F_\infty)=\ZZ^3/\ZZ(1,1,1)=\ZZ^2.
 \label{hinfty}
\end{align}

We next identify the contribution from infinity integrally.
Tracing $\gamma_\infty=\gamma_1^{-1}\gamma_0^{-1}$ gives
\begin{align}
 p(T_{\lambda,\infty})=\ell_\infty(\lambda)
 :=-(A_1\lambda,\lambda)\in\Ker(\mathcal D_1),
 \qquad \lambda\in K_\infty.
 \label{periph}
\end{align}
In particular, $\ell_\infty(\delta)=(-\delta,-\delta)$ and
$\ell_\infty(a)=(b,-a)$. We claim that
\begin{align}
 \frac{\Ker(\mathcal D_1)}{\ell_\infty(K_\infty)}
 =\ZZ\langle[(\delta,0)],[(u_0,0)],[(0,u_1)]\rangle.
 \label{edgebas}
\end{align}
To see this, let $(v_0,v_1)\in\Ker(\mathcal D_1)$.
Modulo $(A_0-I)\Lambda$, the relations $b=-a$ and $\delta=-3a$
give
\begin{align}
 0=[(A_1-I)v_1]
   =(3\psi(v_1)-b^\vee(v_1))[a].
 \label{edgequot}
\end{align}
Since $[a]$ has infinite order in this quotient,
$b^\vee(v_1)=3\psi(v_1)$. Put $n=\psi(v_1)$. Then
$v_1-nu_1\in K_\infty$, so adding
$\ell_\infty(v_1-nu_1)$ changes the second component to $nu_1$.
The kernel condition now puts the first component in
$\Lambda^{A_0}=\ZZ\delta\oplus\ZZ u_0$. Every coset therefore
has a representative $r(\delta,0)+w(u_0,0)+n(0,u_1)$.
It is unique: if such a representative equals
$\ell_\infty(\lambda)$, comparison of the second components gives
$nu_1=-\lambda\in K_\infty$, hence $n=0$ and $\lambda=0$;
then $r=w=0$. This proves \eqref{edgebas} over $\ZZ$.

Let $L\subset H_2(V)$ be generated by
$T_{\delta,\infty},T_{a,\infty}$. By \eqref{periph}, the map
$p$ identifies $L$ with $\ell_\infty(K_\infty)$. Thus
\eqref{graph} gives
\begin{align}
 0\longrightarrow\ZZ s\longrightarrow H_2(V)/L
 \longrightarrow
 \Ker(\mathcal D_1)/\ell_\infty(K_\infty)
 \longrightarrow0.
 \label{quotseq}
\end{align}
The three finite sweep tori lift the basis in \eqref{edgebas}.
Together with $s$, their classes form an integral basis of
$H_2(V)/L$. In particular, this quotient is free and $L$ is primitive.

Sweep the section disc $\sigma(D_\infty)$ by the circle actions corresponding
to $\delta$ and $a$. Since these actions extend over the filling,
this gives relative three-cycles in $(N_\infty,M_\infty^5)$
whose boundaries are $T_{\delta,\infty},T_{a,\infty}$.
Thus $\Image(\tau_\infty)$ contains $L$. By \eqref{hinfty} its
rank is at most two, so primitivity gives
$\Image(\tau_\infty)=L$. This proves \eqref{quotient}, and
\eqref{coker} gives \eqref{fourmap}.
\end{proof}

Thus \eqref{coker} is a map $\ZZ^6\to\ZZ^6$ whose two infinity
columns generate a primitive subgroup. Removing these columns and
their coordinates leaves the map $\ZZ^4\to\ZZ^4$ in~\eqref{fourmap}. 

The invariant form $\xi$ defines a closed real two-form $\Omega$ on
$V$, constant in the flat fiber coordinates and with no horizontal
components. It integrates to one on $s$ and to zero on every sweep
torus. Hence it vanishes on $\Image(\tau_\infty)$ and induces
the first-coordinate functional in the integral basis
$(s,t,t_0,t_1)$ of $Q_\infty$. In particular, $\Omega$ has integral
periods, and its integral determines the first coordinate used below.

\subsection{The finite-filling maps}

For $i=0,1$, use Poincar\'e duality to identify
$H^3(F_i)$ with $H_1(F_i)$. Abelianizing \eqref{pres} gives
\begin{align}
 H_1(F_0)&=\ZZ\langle a,d,g_0\rangle/
                  \ZZ(3r_0a-n_0d+m_0g_0),
 \label{hzero}\\
 H_1(F_1)&=\ZZ\langle a,d,g_1\rangle/
                  \ZZ(2r_1a-n_1d+m_1g_1).
 \label{hone}
\end{align}
Since $\gcd(m_i,n_i) = 1$ by assumption, choose integers $\alpha_i,\beta_i$ with
\begin{align}
 m_i\alpha_i+n_i\beta_i=1,
\end{align}
and define
\begin{align}
 z_i=\alpha_i d+\beta_i g_i.
 \label{zclass}
\end{align}
Then $(a,z_i)$ is an integral basis of $H_1(F_i;\ZZ)$. Indeed, writing $s_0=3,s_1=2$,
the relations give
$d=m_i z_i+s_i r_i\beta_i a$ and
$g_i=n_i z_i-s_i r_i\alpha_i a$.

Given a loop $\gamma$ in $F_i$, take a small disc in the complex normal line to $F_i\subset X$ at each point of $\gamma$. The boundary circles of these discs form a circle bundle over $\gamma$, hence a two-torus in $M_i^5$. We call this torus the \emph{normal circle torus} of $\gamma$, and orient it by the loop first and the positive normal circle second. With the sign convention in \eqref{taudef}, it represents $\partial_i[\gamma]\in H_2(M_i^5)$, and its image in $H_2(V)$ is $\tau_i[\gamma]$.

Let $\varepsilon_i:H_2(M_i^5)\to\Lambda^{A_i}$ be the homomorphism obtained by intersecting a representative two-cycle transversely with a fixed smooth fiber, and put $\mathcal N_i=\sum_{j=0}^{m_i-1}A_i^j$. The resulting fiber one-cycle has monodromy-invariant homology class.
We choose the intersection orientation so that an oriented fiber
loop swept around the positive base circle, with the fiber direction
first, maps to that loop.
\begin{lemma}
\label{tube}
For $\lambda\in\Lambda$, viewed also as a class in $H_1(F_i)$,
\begin{align}
 \varepsilon_i(\partial_i\lambda)&=\mathcal N_i\lambda, \quad
 \int_{\tau_i\lambda}\Omega =-\xi(\lambda,v_i),
 \label{tubel}\\
 \varepsilon_i(\partial_i g_i)&=v_i, \qquad
 \int_{\tau_i g_i}\Omega =0.
 \label{tubeg}
\end{align}
\end{lemma}
\begin{proof}
By \eqref{jfin}, the positive normal circle represents $(-v_i,m_i)$. Since $|s|=1$ on the boundary, the gluing map \eqref{glue} describes its motion in flat coordinates by
\begin{align*}
(x,\theta)\longmapsto(x-u v_i,\theta+m_i u),
\qquad 0\leq u\leq1,
\end{align*}
where increasing $\theta$ by one corresponds to one turn around the base.

For $\lambda\in\Lambda$, the normal circle torus of the corresponding loop is therefore parametrized by
\begin{align*}
(r,u)\longmapsto(r\lambda-u v_i,m_i u),
\qquad 0\leq r,u\leq1.
\end{align*}
This torus represents $\partial_i\lambda\in H_2(M_i^5)$. Cutting it along a reference fiber gives the $m_i$ fiber one-cycles
$\lambda,A_i\lambda,\ldots,A_i^{m_i-1}\lambda$. By the definition of $\varepsilon_i$, we obtain
\begin{align*}
\varepsilon_i(\partial_i\lambda)
=\sum_{j=0}^{m_i-1}A_i^j\lambda
=\mathcal N_i\lambda.
\end{align*}
This proves the first formula in \eqref{tubel}.

The vertical components of the derivatives in the $r$ and $u$ directions are $\lambda$ and $-v_i$, respectively. Since $\Omega$ is given by $\xi$ in the flat fiber coordinates and has no horizontal components, its integral over the normal circle torus is
\begin{align*}
\int_{\tau_i\lambda}\Omega
=\int_0^1\int_0^1\xi(\lambda,-v_i)\,dr\,du
=-\xi(\lambda,v_i).
\end{align*}
This proves the second formula in \eqref{tubel}.

To prove \eqref{tubeg}, recall that $m_i[g_i]=[v_i]$ in $H_1(F_i)$. Applying the homomorphism $\varepsilon_i\circ\partial_i$ and using the first formula in \eqref{tubel} gives
\begin{align*}
m_i\,\varepsilon_i(\partial_i g_i)
=\varepsilon_i(\partial_i v_i)
=\mathcal N_i v_i
=m_i v_i,
\end{align*}
where the last equality holds because $A_i v_i=v_i$. Since $\Lambda^{A_i}$ is torsion-free, we can cancel $m_i$ to obtain $\varepsilon_i(\partial_i g_i)=v_i$. This is the first formula in \eqref{tubeg}.

Finally, applying the homomorphism $c\mapsto\int_{\tau_i c}\Omega$ to the same relation and using the second formula in \eqref{tubel} gives
\begin{align*}
m_i\int_{\tau_i g_i}\Omega
=\int_{\tau_i v_i}\Omega
=-\xi(v_i,v_i)
=0.
\end{align*}
The last equality holds because $\xi$ is alternating. Since $m_i>0$, we obtain $\int_{\tau_i g_i}\Omega=0$, proving the second formula in \eqref{tubeg}.
\end{proof}
\begin{proposition}
\label{matrix}
In the basis $(a,z_0,a,z_1)$ of
$H^3(F_0)\oplus H^3(F_1)$ and basis $(s,t,t_0,t_1)$,
the matrix in \eqref{fourmap} is
\begin{align}
 B=\begin{pmatrix}
 -2n_0&-6r_0\alpha_0&-3n_1&-6r_1\alpha_1\\
 -k_0&r_0\beta_0&2\ell_0&-r_1\beta_1\\
 0&1&0&0\\
 0&0&0&1
 \end{pmatrix},
\label{Bmat}
\end{align}
which satisfies
\begin{align}
 \det B=m_1n_0+m_0n_1=m_0m_1\psi(a_0+a_1).
 \label{fourmat}
\end{align}
\end{proposition}

\begin{proof}
The monodromy matrices give
\begin{align}
 \mathcal N_0a=-k_0\delta,\qquad
 \mathcal N_0d=m_0u_0,\qquad
 \mathcal N_1a=-2\ell_0\delta,\qquad
 \mathcal N_1d=m_1u_1.
 \label{norms}
\end{align}
In the quotient $\Ker(\mathcal D_1)/\ell_\infty(K_\infty)$ appearing in \eqref{quotseq}, we have $[(0,\delta)]=-[(\delta,0)]$, since $\ell_\infty(\delta)=(-\delta,-\delta)$ by \eqref{periph}. Thus the $\delta$ contribution from the filling at $1$ has the opposite $t$ coordinate to the same contribution from the filling at $0$. Using
$\xi(a,u_0)=2$, $\xi(a,u_1)=3$, $\xi(d,\delta)=6$, and
$\xi(d,u_i)=0$, Lemma~\ref{tube} gives
\begin{align}
 \overline\tau_0(a)&=(-2n_0,-k_0,0,0), & \overline\tau_0(d)&=(-6r_0,0,m_0,0), & 
 \overline\tau_0(g_0)&=(0,r_0,n_0,0),\\
\overline\tau_1(a)&=(-3n_1,2\ell_0,0,0), & \overline\tau_1(d)&=(-6r_1,0,0,m_1), & \overline\tau_1(g_1)&=(0,-r_1,0,n_1).
\end{align}
The surjection in \eqref{quotseq} determines the last three coordinates in the basis $(s,t,t_0,t_1)$, and integration against $\Omega$ determines the first.
Substituting \eqref{zclass} gives \eqref{Bmat}.
Expanding along its last two rows gives the first equality in \eqref{fourmat}:
\begin{align}
 \det B=-\det\begin{pmatrix}-2n_0&-3n_1\\-k_0&2\ell_0\end{pmatrix}
       =4\ell_0n_0+3k_0n_1.
 \label{detb}
\end{align}
Since $\psi(\delta)=0$ and $\psi(u_i)=1$, \eqref{logdat} gives $\psi(a_i)=n_i/m_i$, which proves the second equality in~\eqref{fourmat}.
\end{proof}

\begin{corollary}
\label{sphere}
Every choice satisfying \eqref{bezout} gives a complex threefold
$X$ diffeomorphic to $S^6$. Such choices exist for every relatively
prime pair $m_0=3k_0$, $m_1=4\ell_0$.
\end{corollary}

\begin{proof}
Theorem~\ref{simply} gives $\pi_1(X)=0$.
By Proposition~\ref{matrix}, the map \eqref{fourmap} is unimodular,
so $H_2(X)=0$. The regular and bielliptic fibers have Euler
characteristic zero, and the Mumford fiber has Euler characteristic
two~\cite[Proposition~7.1]{EngelArxiv}; hence $\chi(X)=2$.
Integral Poincar\'e duality and the universal coefficient theorem
give $H_4(X)=H_5(X)=0$ and
$H_3(X)\cong\Hom(H_3(X),\ZZ)$. Thus $H_3(X)$ is free, and
$\chi(X)=2$ forces it to vanish.
The simply connected integral homology sphere $X$ is a homotopy
six-sphere by Hurewicz and Whitehead. Since $\Theta_6=0$
\cite[pp.~504--505]{KM}, it is diffeomorphic to $S^6$.
The existence claim follows from Theorem~\ref{simply}.
\end{proof}

The remaining statements in Theorem~\ref{t:main} are proved in the following.  
\begin{proposition}
\label{nonbih}
Let $X=X_{3k,4l}(u)$ and $X'=X_{3k',4l'}(u')$ be threefolds
obtained as in Corollary~\ref{sphere}. If $X$ and $X'$ are
biholomorphic, then $k=k'$ and $l=l'$. Moreover, the local Kuranishi
space of every such $X$ has complex dimension at least one.
\end{proposition}

\begin{proof}
Suppose $\Phi:X\to X'$ is a biholomorphism, write $f_X,f_{X'}$ for
their fibrations, and put $g=f_{X'}\circ\Phi$.
Since $H^2(X;\RR)=0$, the form $g^*\omega_{\mathrm{FS}}$ is exact.
For any regular fiber $T$ of $f_X$, choose a K\"ahler form $\omega_T$.
Stokes' theorem gives
\begin{align}
 \int_T (g|_T)^*\omega_{\mathrm{FS}}\wedge\omega_T=0.
 \label{fibint}
\end{align}
The integrand is nonnegative and is positive wherever $g|_T$
has nonzero differential. Thus $g$ is constant on every regular
fiber. The image of the proper map $(f_X,g)$ is an analytic curve
whose projection to the first factor has degree one. Its normalization
therefore gives a holomorphic map $h:\PP^1\to\PP^1$ with
$g=h\circ f_X$. Applying the same argument to $\Phi^{-1}$ shows that
$h$ is an automorphism.
The local homology monodromies at $0,1,\infty$ have respective orders
$3,4,\infty$, so $h$ fixes each of these points and is the identity.
Fiber multiplicities are preserved, giving $3k=3k'$ and $4l=4l'$.

For the second statement, fix the logarithmic transformation data and vary $u$ in a small disc
about any allowed nonzero value $u_*$. This gives a holomorphic
family, as in~\cite[Remark~8.1]{EngelArxiv}; the same construction applies
to the higher multiplicities. Over a fixed $t\in U$, the three
periods of the covering $M_t^\times$ are independent of $u$, while
the fourth period varies by a nonzero multiple of $\log(u/u_*)$
in the $\CC^\times$ direction. The first three periods span
$\CC^2$, so a complex-linear identification fixing them is the
identity. Consequently the marked torus period map is nonconstant
on every neighborhood of $u_*$.
A fixed complex torus has only countably many integral markings;
hence this holomorphic family of tori cannot be locally isotrivial.
The logarithmic transformations leave the fibers over $U$ unchanged,
and the preceding argument shows that every biholomorphism of the
total spaces preserves $t$. Thus the family of total spaces is
also not locally isotrivial. Its classifying map to the Kuranishi
space of the member at $u_*$ is nonconstant, proving that this space has
dimension at least one.
\end{proof}

\bibliographystyle{amsplain}
\bibliography{Logarithmic_S6}
\end{document}